\documentclass[10pt,a4paper]{amsart}
\usepackage[utf8]{inputenc}
\usepackage[a4paper,centering]{geometry}
\usepackage{amsmath,amssymb} 
\usepackage{mathtools}
\usepackage{dsfont}
\usepackage{color}
\usepackage{lineno}
\usepackage[pdfdisplaydoctitle,colorlinks,breaklinks,urlcolor=blue,linkcolor=blue,citecolor=blue]{hyperref}
\usepackage{yfonts}
\usepackage{appendix}
\usepackage[english]{babel} 

\newtheorem{introtheorem}{Theorem} 
\newtheorem{theorem}{Theorem}[section]
\newtheorem{lem}[theorem]{Lemma}

\newtheorem{kor}[theorem]{Corollary}
\newtheorem{prop}[theorem]{Proposition}
\newtheorem{rem}[theorem]{Remark}
\newtheorem{dfn}[theorem]{Definition}

\numberwithin{equation}{section}

\newcommand{\esssup}{\mathrm{ess~sup}}

\DeclareMathOperator*{\divv}{div}

\newcommand*{\Ascr}{\mathcal A}
\newcommand*{\Bscr}{\mathcal B}
\newcommand*{\Cscr}{\mathcal C}

\newcommand*{\Fscr}{\mathcal F}

\newcommand*{\Lscr}{\mathcal L}
\newcommand*{\Mscr}{\mathcal M}

\newcommand*{\Pscr}{\mathcal P}

\newcommand*{\Sscr}{\mathcal S}

\newcommand*{\N}{\mathbb{N}}

\newcommand*{\R}{\mathbb{R}}

\definecolor{orange}{rgb}{1.0, 0.55, 0.0} 

\newcommand{\loc}{\text{loc}}
\newcommand{\uloc}{\text{uloc}}
\newcommand{\uno}{\mathbf{1}}

\begin{document}
  \title[$2D$ vorticity Euler equations]{$2D$ vorticity Euler equations: Superposition solutions and nonlinear Markov processes}
  \author{Marco Rehmeier}
    \address{Institute of Mathematics, TU Berlin, 10623 Berlin, Germany}
    \email{rehmeier@tu-berlin.de}
  \author{Marco Romito}
	\address{Dipartimento di Matematica, Università di Pisa, largo B. Pontecorvo 5, 56127, Pisa, Italia}
    \email{marco.romito@unipi.it}
  \dedicatory{In memory of Giuseppe Da Prato}
  \date{}
  \begin{abstract}
	In this note we contribute two results to the theory of the $2D$ Euler equations in vorticity form on the full plane. First, we establish a generalized Lagrangian representation of weak (in general measure-valued) solutions, which includes and extends classical results on the Lagrangianity of weak solutions. Second, we construct nonlinear Markov processes which are uniquely determined by a selection of weak solutions from initial data in $L^1\cap L^p$, $p \geq 2$, and related spaces such as the classical and uniformly localized Yudovich space. It is well-known that for $p <\infty$ weak solutions are in general not unique, which renders a suitable selection nontrivial. 
  \end{abstract}
  \maketitle
	\section{Introduction}
	We consider the $2D$ Euler Equations in vorticity form
	\begin{equation}\label{eq1}\tag{EE}
		\partial_t \omega + (v\cdot \nabla ) \omega = 0,\quad \omega(0)=\omega_0,\quad v = K* \omega,
	\end{equation}
	on $\R^2$,
	where $K$ is the \emph{Biot--Savart kernel}
	\begin{equation}\label{BS}
		K(x) := \frac{x^\perp}{2\pi|x|^2}, \quad (x_1,x_2)^\perp := (-x_2,x_1),\quad x=(x_1,x_2) \in \R^2,
	\end{equation}
	and the convolution $K * \omega$ occurs only in the spatial variable. These are the fundamental fluid dynamical equations for the vorticity $\omega$ associated with the velocity field $v$ of an inviscid incompressible fluid.
	Since $\divv v =0$, setting $K(\omega)(t,x):= (K*\omega(t))(x)$ the first equation in \eqref{eq1} can be written as
	\begin{equation}\label{eq2}
		\partial_t \omega + \divv (K(\omega)\omega) = 0,
	\end{equation}
	which is a first-order nonlinear Fokker--Planck equation (FPE). Generally, FPEs are equations for measures, and we will insist on this viewpoint for \eqref{eq1}. The purpose of this note is to contribute new results to two distinct aspects of the theory of \eqref{eq1}. Let us briefly describe these purposes.
	
	\textbf{Superposition solutions.}  Weak (by which we mean distributional, sometimes also called very weak) solutions to \eqref{eq1} are unique when the initial vorticity $\omega_0$ is bounded or nearly so \cite{Y63,CS21}. Otherwise, uniqueness is an open problem. In such potentially ill-posed cases it is desirable to identify particularly reasonable solutions, either from a physical or mathematical viewpoint. One such class are \emph{Lagrangian solutions}
	$$\omega(t,x) = \omega_0(X^{-1}(t,x)),$$
	where $X: [0,\infty)\times \R^2 \to \R^2$ is a measure-preserving flow of the associated ODE 
	\begin{equation}\label{eq_intro}
		\dot{y}(t) = K(\omega(t))(y(t)), \quad t \geq 0,
	\end{equation}
	see Definition \ref{def:Lagrang-sol}. This ODE models particle trajectories associated with \eqref{eq1}. Several results prove the Lagrangianity of weak solutions in different classes of solutions and initial data, summarized in Propositions \ref{prop:p>2} and \ref{prop:ES+VS}. Despite these important results, it remains an open question whether general weak solutions to \eqref{eq1}, in particular low-integrability or measure-valued solutions, are Lagrangian. In particular, for measure-valued initial data, solutions typically remain measure-valued, due to a lack of regularizing effects (for instance vortex sheets, see \cite{SSBF1981,DucRob1988}, or point vortices, see \cite{MarPul1994}). It cannot be expected that such solutions are Lagrangian in the above sense. Our contribution in this direction is the following representation of weak solutions as Lagrangian solutions in a generalized sense.
	
	\begin{introtheorem}[see Thm. \ref{thm1} for details]
		Any nonnegative finite measure-valued weak solution $\omega$ to \eqref{eq1}, satisfying
		\begin{equation}\label{L1-int-sol}
			\frac{(K* \omega(t)) (x)}{1+|x|}  \in L^1\big((0,T)\times \mathbb{R}^2;d\omega(t)dt\big),\quad \forall T>0,
		\end{equation}
		is a \emph{superposition solution}, i.e. there is a nonnegative finite Borel measure $\eta$ on $C([0,\infty),\R^2)$ (the space of continuous functions from $[0,\infty)$ to $\R^2$) concentrated on solution curves to \eqref{eq_intro} such that
		\begin{equation*}
			\eta \circ \pi_t^{-1} = \omega(t), \quad \forall t \geq 0,
		\end{equation*}
		where $\pi_t: C([0,\infty),\R^2)\to \R^2$, $\pi_t(f )=f(t)$, denotes the canonical projection at time $t$, and $\eta \circ \pi_t^{-1}$ is the image measure of $\eta$ with respect to $\pi_t$.
	\end{introtheorem}
	
	Any Lagrangian solution is a superposition solution, with $\eta = \omega_0\circ X^{-1}$, but not vice versa. Thus superposition solutions are generalized Lagrangian solutions, and 
	Theorem 1 contains and extends the previously known results on the Lagrangianity of weak solutions summarized in Section \ref{sect:SP-sol}. 
	
	To compare Lagrangian and superposition solutions, consider the family of disintegration kernels $\{\eta_x\}_{x\in \R^2}$ of $\eta$ with respect to $\pi_0$, i.e the $\eta\circ \pi_0^{-1}$-a.s. unique measurable (in $x$) family of nonnegative Borel measures on $C([0,\infty),\R^2)$ such that $\eta_x(\pi_0\neq x)=0$ and $\eta = \int_{\R^2}\eta_{x}\,d(\eta\circ \pi_0^{-1})(x)$. For a Lagrangian solution, one has $\eta_x = \delta_{y_x}$, where $y_x$ denotes the member of the associated flow of \eqref{eq_intro} with initial datum $y_x(0) = x$. For superposition solutions, the measures $\eta_x$ need not be Dirac, but may have a larger support on the set of solution curves to \eqref{eq_intro} with initial datum $x$. In this case, the initial vorticity is not transported along a family of separated flow trajectories, but along a superposition of trajectories, of which several are initiated from the same initial value. This intuition is supported by Corollary \ref{cor:prob-repr}.
	Our result appears to be quite general, since the integrability assumption \eqref{L1-int-sol} is rather mild. Indeed, the minimal assumption needed to define weak solutions is $K*\omega\in L^1_\loc(\R^2\times(0,\infty),d\omega(t)\,dt)$, see Definition \ref{def:weak-sol}.

	\textbf{Nonlinear Markov processes.}
	The probability measure-valued solution family of any well-posed linear FPE can be represented as the one-dimensional time marginals of a uniquely determined Markov process, which consists of the unique solution path laws of the associated stochastic differential equation (SDE). This is not true for nonlinear FPEs, since the former relation is based on the Chapman--Kolmogorov equations, which are not satisfied for solution families to nonlinear FPEs, even if the latter is well-posed. Inspired by McKean \cite{McKean-classical2}, in \cite{R./Rckner_NL-Markov22} the first-named author and M. Röckner studied a notion of \emph{nonlinear Markov processes}, which is tailored to lead to a similar connection to nonlinear FPEs and the corresponding stochastic equations as in the linear case. Indeed, the main result in \cite{R./Rckner_NL-Markov22} states that for suitable solution families to (not necessarily well-posed) nonlinear FPEs there exists a uniquely determined nonlinear Markov process consisting of solution path laws to the associated distribution-dependent SDE with marginals given by the nonlinear FPE solutions. This result applies to many nonlinear PDEs (a subclass of nonlinear FPEs), such as porous media and Burgers equations. Hence, solutions to these equations have a (probabilistic) representation as the one-dimensional time marginals of a nonlinear Markov process. In \cite{BRZ23}, this  was also proven for the $2D$ vorticity Navier--Stokes equations
	$$\partial_t \omega + (v\cdot \nabla)\omega - \Delta \omega = 0, \quad v = K*\omega.$$
	We give a brief introduction to nonlinear Markov processes in Section \ref{subsect:NL-MP}. 
	
	In the light of these results, it is natural to ask whether solutions to \eqref{eq1} fit this framework and can be related to nonlinear Markov processes. Our second main purpose of this note is to answer this question affirmatively for several classes of initial data and solutions to \eqref{eq1}.
	
	\begin{introtheorem}[see Thm. \ref{thm:main-appl-selection} and Prop. \ref{prop:well-posed-case} for details]
		Let $E$ be the subset of probability densities from one of the following spaces: $(L^1\cap L^\infty)(\R^2)$, the classical or uniformly localized Yudovich spaces $Y^\Theta$ or $Y^\Theta_{\uloc}$, or $(L^1\cap L^p)(\R^2)$, with $p\geq2$. There exists a family of weak solutions $\{\omega^\zeta\}_{\zeta \in E}$ to \eqref{eq1}, $\omega^\zeta \in L^\infty_\loc([0,\infty),E)$ and $\omega^\zeta(0) = \zeta$, such that $\{w^\zeta_t\}_{\zeta \in E, t \geq 0}$ are the one-dimensional time marginals of a uniquely determined nonlinear Markov process. The latter consists of path laws $\mathbb{P}_\zeta, \zeta \in E,$ concentrated on solutions to \eqref{eq_intro}.
	\end{introtheorem}
	
	Since solutions from $(L^1 \cap L^p)(\R^2)$-initial data with $p<\infty$ are not expected to be unique (see \cite{Vis2018a,Vis2018b,DlBACGJK2024} for rigorous results with suitable forcing), in this case our result requires an appropriate selection of solutions, see Proposition \ref{prop:Romitos-prop}.
	
	Opposed to previous applications of the results from \cite{R./Rckner_NL-Markov22}, \eqref{eq1} is a first-order equation. Hence, the corresponding nonlinear Markovian dynamics is deterministic. We anticipate a substantial theory of nonlinear Markov processes in future works, including results on the associated semigroups, invariant measures and further asymptotic properties. We expect that, via Theorem 2, this nonlinear theory will contribute to the theory of \eqref{eq1} and, vice versa, that Theorem 2 will be a further motivation to achieve progress in this direction.
	
    Finally, we conclude the introduction with a remark about the three-dimensional case. While it may be tempting to apply the strategy illustrated in this paper to  three dimensions, doing so does not appear to be straightforward. First, the equation for vorticity is vector-valued, making it unclear how to interpret it as a nonlinear Fokker-Planck equation. Furthermore, the physics of the problem changes dramatically in three dimensions. The extra stretching term is the key mathematical object that captures the new physical phenomena (see for instance \cite{Fri1995}). For one thing, vorticity is not conserved along the Lagrangian trajectories. Representation formulas must account for this effect through the cumulative effect of the deformation tensor. This can be seen, for example, in \cite{BusFlaRom2005,ConIye2008} (although these papers deal with noisy Lagrangian trajectories, leading to the Navier-Stokes equations).
	
	\textbf{Organization.} In Section \ref{sect:FPE}, we briefly recall nonlinear FPEs, the associated distribution-dependent SDEs, their relation, and pose \eqref{eq1} in this FPE-framework. Section \ref{sect:SP-sol} contains our results on superposition solutions, and in Section \ref{sect:NL-MP} we present our results related to nonlinear Markov processes.
	
	\paragraph{Notation.} 
	We write $B_r(x)\subseteq \R^d$ for the Euclidean ball of radius $r>0$ with center $x\in \R^d$. $\Mscr^+_b(\R^d)$ denotes the set of nonnegative finite Borel measures on $\R^d$ while, for a topological space $X$, $\Pscr(X)$ is the set of Borel probability measures on $X$. If $X = \R^d$, we write $\Pscr = \Pscr(\R^d)$.
	
	For $k \in \N_0 \cup \{\infty\}$ and $U \subseteq \R^d$, $C_{(c)}^k(U,\R^m)$ are the usual spaces of continuous (compactly supported) functions from $U$ to $\R^m$ with continuous partial derivatives up to order $k$. If $k=0$, we write $C_{(c)}(U,\R^m)$. If $U=[0,\infty)$, we denote by $\pi_t$ the canonical projection $\pi_t(f) = f(t)$, $t \geq 0$, on $C([0,\infty),\R^m)$.
	
	For $p \in [1,\infty]$ and a Borel measure $\mu$, $L_{(\textup{loc})}^p(\R^d,\R^m;\mu)$ are the usual spaces of $L^p$-integrable functions from $\R^d$ to $\R^m$, with their usual norms $||\cdot||_{L^p}$. If $\mu = dx$ (Lebesgue measure), we abbreviate by $L^p(\R^d,\R^m)$. If $m=1$, we write $L_{(\textup{loc})}^p(\R^d;\mu)$. $W_{(\textup{loc})}^{1,p}(\R^d,\R^m)$ are the standard (local) first-order Sobolev spaces with norm $||\cdot||_{W^{1,p}}$. We write $W^{1,p}_{(\textup{loc})}(\R^d)$ if $m=1$.
	
	We write $C_\textup{w}([0,\infty),L^p(\R^2))$ for the space of continuous maps from $[0,\infty)$ into $L^p(\R^2)$, the latter  considered with its weak topology.
	\section{Nonlinear Fokker--Planck equations}\label{sect:FPE}
	Nonlinear Fokker--Planck equations (FPE) are parabolic equations for Borel measures on $\R^d$ of type
	\begin{equation}\label{FPE}
		\partial_t \mu_t = \partial^2_{ij}\big(a_{ij}(t,\mu_t,x)\mu_t\big) - \partial_i \big(b_i(t,\mu_t,x)\mu_t\big),\quad \mu_0 = \zeta,\quad t \geq 0,
	\end{equation}
	with $a_{ij}, b_i: [0,\infty) \times \Mscr^+_b(\R^d) \times \R^d \to \R$, $1\leq i,j \leq d$, $\zeta \in \Mscr_b^+(\R^d)$. $(a_{ij})_{1\leq i,j \leq d}$ is always assumed to be pointwise symmetric and nonnegative definite. The notion of solution we consider is in the usual distributional sense, also called very weak solution in the literature. Precisely, a curve $(\mu_t)_{t\geq 0}$ in $\Mscr^+_b$ is a solution to \eqref{FPE}, if $t\mapsto \mu_t(A)$ is Borel measurable for each $A \in \Bscr(\R^d)$, $[(t,x)\mapsto a_{ij}(t,\mu_t,x)], [(t,x)\mapsto b_i(t,\mu_t,x)] \in L^1_\loc([0,\infty)\times \R^d;d\mu_t dt)$ and 
	$$\int_{[0,\infty)\times \R^d} \partial_t \varphi(t,x) + a_{ij}(t,\mu_t,x) \partial_{ij}\varphi(t,x) + b_i(t,\mu_t,x)\partial_i \varphi(t,x) \, d\mu_t(x) dt + \int_{\R^d} \varphi(0,x)\,d\zeta(x) = 0$$
	for all $\varphi \in C^\infty_c([0,\infty)\times \R^d)$.
	If each $\mu_t$ is a probability measure, $\mu$ is called probability solution.
	\eqref{eq2} is a first-order equation of this type, with $d=2$ and
	\begin{equation*}
		a_{ij} = 0,\quad b_i(t,\mu,x) = K_i(\mu)(x), \quad i,j \in \{1,2\},
	\end{equation*}
	with $K = (K_1,K_2)$ as in \eqref{BS}, and where we write $K(\mu)(x):= \int_{\R^2}K(x-y)\,d\mu(y)$, which is consistent with the notation $K(f)$ for functions $f:\R^2\to\R$.
	There is a natural relation between \eqref{FPE} and its associated McKean--Vlasov stochastic differential equation 
	\begin{equation}\label{DDSDE}
		dX_t = b(t,\mathcal{L}(X_t),X_t)dt + \sqrt{2}\sigma(t,\mathcal{L}(X_t),X_t)dB_t,\quad \mathcal{L}(X_0) = \zeta,\quad t \geq 0,
	\end{equation}
	where $B$ is a $d$-dim. Brownian motion, $\sigma \sigma^T = (a_{ij})_{i,j \leq d}$, and $\mathcal{L}(X_t)$ denotes the distribution of $X_t$. For \eqref{eq2}, this equation is
	\begin{equation*}\label{ODE}
		\dot{X}_t = K(\mathcal{L}(X_t))(X_t),
	\end{equation*}
	i.e. an ODE with a distribution-dependent vector field.
	
	\begin{rem}\label{rem:SP-princ}
		More precisely, the relation between \eqref{FPE} and \eqref{DDSDE} is the following. For any weakly continuous probability solution $t\mapsto \mu_t$ such that
		$$\bigg[(t,x)\mapsto \frac{|a(t,\mu_t,x)|+| b(t,\mu_t,x)\cdot x|}{1+|x|^2}\bigg] \in L^1((0,T)\times \R^d;d\mu_t dt),\quad \forall T>0,$$
		there exists a martingale solution $Q \in \Pscr\big(C([0,\infty),\R^d)\big)$ to the nonlinear martingale problem associated with $a$ and $b$ such that $Q\circ \pi_t^{-1} = \mu_t$; equivalently, there is a probabilistically weak solution process $X$ to \eqref{DDSDE} with $\mathcal{L}(X_t) = \mu_t$, see \cite{BR18,BR18_2,BRS19-SPpr}. This is the nonlinear extension of the Ambrosio--Figalli--Trevisan superposition principle \cite{Ambrosio08,Figalli09,trevisan16}. Here $x\cdot y$ denotes the usual Euclidean inner product.
	\end{rem}
	\begin{rem}\label{rem:mg-sol-det-case}
		In the first-order case $a_{ij} = 0$, solutions to the nonlinear martingale problem are exactly those Borel measures $Q$ on $C([0,\infty),\R^d)$ that are concentrated on absolutely continuous integral solution curves of the ODE
		$$\dot{y}(t) = b(t,Q\circ \pi_t^{-1},y(t)),\quad t \geq 0.$$
	\end{rem}

	\section{Superposition solutions}\label{sect:SP-sol}
	We start by collecting the following well-known properties of the Biot--Savart kernel $K$ from \eqref{BS}. For $p \in [1,\infty]$, we denote by $p' \in [1,\infty]$ its conjugate, i.e. $\frac 1 p + \frac 1 {p'} = 1$.
	\begin{lem}\label{lem:K-prop}
		\begin{enumerate}
			\item [(i)] $K \in L^1(\R^2,\R^2) + L^\infty(\R^2,\R^2)$.
			\item[(ii)] Let $p> 1$ and $\omega \in (L^1 \cap L^p)(\R^2)$. Then $K*\omega \in W^{1,p}_{\textup{loc}}(\R^2,\R^2)$. If $p \geq \frac 4 3$, then $K*\omega \in L_\loc^{p'}(\R^2,\R^2)$, and $K*\omega \in L^{p'}(\R^2,\R^2)$ if also $p<2$.
			\item[(iii)] Let $p\geq 2$ and $\omega \in (L^1 \cap L^p)(\R^2)$. Then $(K*w)w \in L^1(\mathbb{R}^2,\mathbb{R}^2)$. If $\omega \in L^\infty([0,\infty);(L^1\cap L^p)(\R^2))$, then $(K*\omega)\omega \in L^\infty([0,\infty);L^1(\R^2,\R^2))$.
		\end{enumerate}
	\end{lem}
	\begin{proof}
		\begin{enumerate}
			\item [(i)] $K = K_1+K_2 := \mathds{1}_{B_1(0)}K + \mathds{1}_{B_1(0)^c}K \in L^1(\R^2,\R^2) + L^\infty(\R^2,\R^2)$, where $B_1(0)\subseteq \R^2$ denotes the Euclidean ball with radius $1$ centered at $0$.
			\item[(ii)] For the first part, see \cite[(1.5), (2.11)]{BRZ23}.
			For the second part, since $\omega \in L^{\frac 4 3}(\R^2)$, by (i) and \cite[(1.5)]{BRZ23} we have $K*\omega \in L^4(\R^2,\R^2)$, i.e. in particular $K*\omega \in L^q_\loc(\R^2,\R^2)$ for all $q \in [1,4]$. Since $p\geq \frac 4 3$ implies $p' \in [1,4]$, the claim follows. Finally, the third part follows directly from \cite[(1.5)]{BRZ23}.
			\item[(iii)] Young inequality yields $K_1*\omega \in (L^1 \cap L^p)(\R^2,\R^2)$ and $K_2*\omega \in L^\infty(\R^2,\R^2)$, so (since $p\geq 2$) in particular $K*\omega \in L^{p'}(\R^2,\R^2) + L^\infty(\R^2,\R^2)$. Both claims now follow from the assumption on $\omega$.
		\end{enumerate}
	\end{proof}

	As said in the introduction, we consider \eqref{eq1} as the first-order nonlinear FPE \eqref{eq2} with the following standard notion of (in general measure-valued), solutions.
	\begin{dfn}\label{def:weak-sol}
		A curve $\omega: (0,\infty)\to \mathcal{M}^+_b(\R^2)$ is a \textit{weak solution} to \eqref{eq1} with initial datum $\omega_0 \in \mathcal{M}^+_b(\R^2)$, if
		\begin{enumerate}
			\item[(i)] $t\mapsto \omega(t)(A)$ is Borel measurable for all $A\in \Bscr(\R^2)$.
			\item [(ii)] $K * \omega \in L^1_{\text{loc}}\big((0,\infty)\times \R^2;d\omega(t)dt\big)$
			\item [(iii)] The identity
			\begin{equation*}\label{eq:euler-vort-weak}
				\int_0^\infty\int_{\R^2} \partial_t\varphi + \nabla \varphi \cdot (K*\omega(t))\,d\omega(t)dt +\int_{\R^2} \varphi(0) \,d\omega_0= 0
			\end{equation*}
			holds for all $\varphi \in C^\infty_c([0,\infty)\times \R^2)$.
		\end{enumerate}
	\end{dfn}
	If a solution consists of measures $\omega(t)$ absolutely continuous w.r.t. Lebesgue measure, we identify $\omega(t)$ with its density (also denoted $\omega(t)$ if no confusion can occur). Statements such as $\omega(t) \in L^p(\R^2)$ and $\omega \in L^\infty((0,\infty);L^p(\R^d))$ are understood accordingly.
	
	\begin{rem}\label{rem:int-Sobolevembedd}
		If $\omega \in L_{\textup{loc}}^1((0,\infty);(L^1\cap L^p)(\R^2))$ for $p\geq \frac 4 3$, then Lemma \ref{lem:K-prop} implies $K*\omega \in L_{\textup{loc}}^1((0,\infty),L^{p'}_{\textup{loc}}(\R^2))$, and (ii) of the previous definition is satisfied. In particular, $\omega$ is a standard weak solution of \eqref{eq1} in the sense of (for instance) \cite{Che1995}.
	\end{rem}
	
	\begin{lem}\label{l:euler-vort-weak2}
		Let $\omega$ be a weak solution to \eqref{eq1} such that $K*\omega \in L^1_{\textup{loc}}([0,\infty)\times \R^2;d\omega(t)dt)$.
		\begin{enumerate}
			\item [(i)] 	If $t \mapsto \omega(t)$ is vaguely continuous, then 
			\begin{equation}\label{aux123}
				\int_{\R^2} \phi\, \omega(t)dx = \int_{\R^2} \phi \,w_0 dx + \int_0^t \int_{\R^2} \nabla \phi \cdot (K*\omega(s))\,d\omega(s)ds,\quad \forall t \geq 0,
			\end{equation}
			holds for all $\phi \in C^\infty_c(\R^2)$ and is, in fact, equivalent to (iii) of the previous definition. More precisely, a vaguely continuous curve $(\omega(t))_{t>0}$ in $\Mscr^+_b$ satisfying both $K*\omega \in L^1_{\textup{loc}}([0,\infty)\times \R^2;d\omega(t)dt)$ and \eqref{aux123} satisfies Definition \ref{def:weak-sol}.
			\item[(ii)]
			
			There exists a unique vaguely continuous $dt$-version $\tilde{w}$ of $\omega$ on $[0,\infty)$ with initial datum $\omega_0$. If in addition \eqref{L1-int-sol} holds, then $\tilde{\omega}(t)(\R^2) = \omega_0(\R^2)$ for all $t >0$, and $\tilde{\omega}$ is weakly continuous.
			\item[(iii)] Let $p\geq \frac 4 3$. If $w_0 \in (L^1\cap L^p)(\R^2)$ and $w \in L^\infty([0,\infty);(L^1\cap L^p)(\R^2))$, then \eqref{L1-int-sol} holds and the weakly continuous version of $\omega$ additionally belongs to $C_{\textup{w}}([0,\infty);L^p(\R^2))$.
		\end{enumerate}
	\end{lem}
	\begin{rem}
		In particular, any weak solution $\omega$ with initial datum $\omega_0 \in \Pscr(\R^2)$ satisfying $w \in L^\infty([0,\infty);(L^1\cap L^p)(\R^2))$ for $p \geq \frac 4 3$ has a $dt$-version in $C_{\textup{w}}([0,\infty);L^p(\R^2))$ which is a weakly continuous probability solution satisfying \eqref{aux123}. Moreover, for $w \in L^\infty([0,\infty),L^1 \cap L^p)$ weakly continuous, being a weak solution is equivalent to \eqref{aux123}.
	\end{rem}
	\begin{proof}
		\begin{enumerate}
			\item [(i)] The proof is standard, see for instance \cite[Prop.6.1.2]{FPKE-book15}.
			\item[(ii)] The first part follows from \cite[Lem.2.3]{Rehmeier-nonlinear-flow-JDE}. Since $\tilde{\omega}$ satisfies \eqref{aux123}
			for any $\phi \in C^\infty_c(\R^2)$, it suffices to let $\phi = \phi_k$ such that $\phi_k= 1$ on $B_k(0)$ and $|\nabla \phi_k(x)|\leq (1+|x|)^{-1}$, and let $k \to \infty$ to conclude $\tilde{\omega}(t)(\R^2) = \omega_0(\R^2)$ for all $t\in [0,\infty)$ by Lebesgue's dominated convergence, which applies due to \eqref{L1-int-sol}. Since every vaguely continuous curve in $\Mscr^+_b$ with constant total mass is weakly continuous, the proof is complete.
			
			\item[(iii)] \eqref{L1-int-sol} follows from Lemma \ref{lem:K-prop}. The weak $L^p$-continuity can be proven as follows. We abbreviate $\omega(t,f) := \int_{\R^2} f(x)\omega(t,x)dx$. Let $g\in L^{p'}(\R^2)$, $t \geq 0$, $\varepsilon>0$ and $\phi \in C_c(\R^2)$ such that $||g-\phi||_{{L^p}'} < \varepsilon$. Then, for all $r \geq 0$, $|\omega(r,\phi)-\omega(r,g)|\leq \varepsilon||\omega||_{L^\infty(L^{p})}$. Since 
			$$\omega(t,g)-\omega(s,g) = \omega(t,g)-\omega(t,\phi) + \omega(t,\phi) - \omega(s,\phi) + \omega(s,\phi) - \omega(s,g),\quad \forall s \geq 0,$$
			we find
			$$\limsup_{s\to t}|\omega(t,g)-\omega(s,g)| \leq 2\varepsilon ||\omega||_{L^\infty(L^p)} + \lim_{s\to t}|\omega(t,\phi)-\omega(s,\phi)|,$$
			and the claim follows by letting $\varepsilon \to 0$.
		\end{enumerate}
	\end{proof}
	
	Several works are concerned with the question whether solutions to \eqref{eq1} are related to the distribution-dependent ODE \eqref{eq_intro}, more precisely
	whether they are \textit{Lagrangian} in the following sense.
	\begin{dfn}\label{def:Lagrang-sol}
		$\omega: (0,\infty)\times \R^2\to \R$ is a \textit{Lagrangian solution} to \eqref{eq1} with initial datum $\omega_0$, if 
		$$\omega(t,x) = \omega_0(X^{-1}(t,x)),\quad \forall t \geq 0,$$
		where $X^{-1}$ is the inverse of a measure preserving flow $X: \R\times \R^2 \to \R^2$ to \eqref{eq_intro}, i.e. for $dx$-a.e. $x \in \R^2$, $t\mapsto X(t,x)$ is an absolutely continuous integral solution to \eqref{eq_intro} with $X(0,x) = x$, and $dx \circ X(t)^{-1} = dx$ for all $t \in \R$.
	\end{dfn}
	$X$ as in the previous definition is called \textit{regular Lagrangian flow}.
	In \cite{DiPL89}, the notion of \textit{renormalized solution} for linear transport equations was introduced and the following relation  with Lagrangian solutions was proven. Let $b: [0,\infty)\times  \R^2 \to \R^2$.
	\begin{prop}\label{prop:DiPerna-L}
		Let $p\geq 1$ and $\omega_0 \in L^p(\R^2)$. If $\divv b \equiv 0$,
		\begin{equation}\label{0}
			I) \,\,b \in L^1_{\textup{loc}}([0,\infty),W^{1,1}_{\textup{loc}}(\R^2)) \,\,\textup{ and }\,\,II)\,\, \frac{b}{1+|x|} \in L^1_{\textup{loc}}([0,\infty);(L^1+L^\infty)(\R^2)),
		\end{equation}
		then there exists a unique renormalized solution $\omega \in L_{\textup{loc}}^\infty([0,\infty),L^p(\R^2))$ to the linear PDE
		\begin{equation}\label{1}
			\partial_t \omega + \divv(b\omega) = 0,\quad (t,x) \in (0,\infty)\times \R^2,
		\end{equation}
		with initial datum $\omega_0$, and $\omega$ is a Lagrangian solution to \eqref{1}. The latter is understood as in Definition \ref{def:Lagrang-sol}, but with a flow for the ODE associated with \eqref{1} instead of \eqref{eq_intro}.
	\end{prop}
	Since a weak solution $\omega \in L^\infty_{\textup{loc}}([0,\infty),L^p(\R^2))$ to \eqref{eq1} is also a weak solution to \eqref{1} with $b = K(\omega)$, the previous proposition implies the Lagrangianity of $\omega$, if \eqref{0} holds for $b = K(\omega)$ and if $\omega$ is a renormalized solution to \eqref{1} with $b = K(\omega)$ (note that $\divv K(\omega) \equiv 0$ holds, since $K$ is the Biot--Savart kernel).
	Both is true for $p\geq 2$, which leads to the following result, proven in \cite{L06_ES>1}.
	\begin{prop}\label{prop:p>2}
		Let $\omega_0 \in L^p(\R^2)$, $p \geq 2$. Then any weak solution $\omega \in L^\infty_{\textup{loc}}([0,\infty),L^p(\R^2))$ to \eqref{eq1} with initial datum $\omega_0$ is Lagrangian (in the sense of Definition \ref{def:Lagrang-sol}).
	\end{prop}
	\begin{rem}
		The previous result does not yield uniqueness of weak or renormalized solutions to the nonlinear equation \eqref{eq1}.
	\end{rem}
	For weak solutions in $L^p(\R^2)$ with $1 \leq p < 2$, or -- more generally -- for measure-valued initial data and solutions, there seems to be no general result of Lagrangianity. However, several special case results for low-integrability solutions exist. 
	Two frequently studied solution classes consist of those solutions obtained as limits of sequences $\omega_n$ with smooth initial data $\omega_n^0$, $n\geq 1$, converging to $\omega_0$, where $\omega_n$ are either exact solutions (ES) to \eqref{eq1} or to the Navier--Stokes equations with vanishing viscosity $\nu_n \to 0$ (VV). We call (ES)- and (VV)-solutions those weak solutions to \eqref{eq1} obtained as such limits, respectively. Without further details, we summarize results on the Lagrangianity of such solutions, as obtained in \cite{L06_ES>1,BBC16_ES=1,CS15_VV>1,CNSS17_VV=1}, as follows.
	
	\begin{prop}\label{prop:ES+VS}
		Let $ 1\leq p <2$, and $\omega_0 \in L^p(\R^2)$. Then there exist (ES)- and (VV)-solutions to \eqref{eq1} which are Lagrangian.
	\end{prop}
	Beyond these very interesting special results, the Lagrangian property seems open (and cannot be expected) for general low-integrability weak solutions or measure-valued solutions and measure-initial data. The following main result of this section establishes a general representation of weak solutions as Lagrangian solutions in a suitably generalized sense.
	\begin{theorem}\label{thm1}
		Let $\omega :(0,\infty) \to \mathcal{M}^+_b(\R^2)$ be a weak solution to \eqref{eq1} with initial datum $\omega_0 \in \Mscr^+_b(\R^2)$, such that $\esssup_{t\in (0,T)}\omega(t)(\R^2) < \infty$ for all $T>0$ and \eqref{L1-int-sol} holds.
		Then there is a weakly continuous $dt$-version $\tilde{\omega}$ of $\omega$ which is a superposition solution, i.e. there is a nonnegative finite Borel measure $\eta$  on $C([0,\infty),\R^2)$, concentrated on absolutely continuous integral solutions to \eqref{eq_intro} with $\tilde{\omega}$ instead of $\omega$,
		such that 
		\begin{equation}\label{111}
			\eta \circ \pi_t^{-1} = \tilde{\omega}(t), \quad \forall t \geq 0.
		\end{equation}
	\end{theorem}
	\begin{rem}
		If $\omega$ is weakly continuous, then $\omega = \tilde{\omega}$. Note that \eqref{111} holds also $dt$-a.s with $\omega$ instead of $\tilde{\omega}$ and that \eqref{eq_intro} has the same integral solutions for $\tilde{\omega}$ and $\omega$.
	\end{rem}
	The proof uses the superposition principle for continuity equations and its nonlinear extension. As said in Section \ref{sect:FPE}, the latter was proven in the more general case of second-order Fokker--Planck equations \cite{BR18,BR18_2,BRS19-SPpr}, but, as it seems to us, has not explicitly been applied to the first-order case.
	\begin{proof}[Proof of Theorem \ref{thm1}]
		Assume first that $\omega$ as in the assertion is weakly continuous. Then the assertion follows from the nonlinear superposition principle \cite[Sect.2]{BR18_2} and Remark \ref{rem:mg-sol-det-case}, applied under the relaxed assumptions of \cite[Thm.1.1]{BRS19-SPpr}, with $\omega = \tilde{\omega}$. For the general case, it suffices to prove that $\omega$ as in the assertion has a weakly continuous $dt$-version to which the first part of the proof can be applied. This follows from Lemma \ref{l:euler-vort-weak2} (ii).
	\end{proof}
	\begin{rem}
		The proof of Lemma \ref{l:euler-vort-weak2} (ii) suggests that assumption \eqref{L1-int-sol} is essentially sharp in the sense that if $1+|x|$ in \eqref{L1-int-sol} is replaced by $1+|x|^\alpha$ for $\alpha>1$, then in order to repeat the proof, one would need a sequence $(\varphi_k)_{k \in \N} \subseteq C^1_c(\R^2)$ such that $|\nabla \varphi_k(x)|\leq (1+|x|^\alpha)^{-1}$ for all $x$ and $\varphi_k = 1$ on $B_k(0)$. But such a sequence does not exist, since these assumptions on $\varphi_k$ and $\nabla \varphi_k$ seem incompatible with the compact support assumption, at least for large $k$. Indeed, for such a sequence we would have $\varphi_k((0,k)) = 1$ and $\varphi_k((0,N_k)) = 0$ for some $k< N_k \in \N$ (here $(0,k)$ and $(0,N_k)$ denote points in $\R^2$ on the $x_2-$axis). We find
        \begin{align*}
            1 = |\varphi_k((0,k)) - \varphi((0,N_k))| \leq \sum_{\ell = k}^{N_k-1}|\varphi_k((0,\ell)) - \varphi((0,\ell+1))| \leq \sum_{\ell \geq k} \ell^{-\alpha}
        \end{align*}
        for all $k \in \N$, which requires $\alpha \leq 1$. For the final inequality we used the mean value theorem, Cauchy--Schwarz inequality, the assumed pointwise estimate for each $|\nabla \varphi_k|$ and $\max_{x \in (\ell, \ell+1)} (1+|x|^\alpha)^{-1} = (1+|\ell|^\alpha)^{-1}$.
	\end{rem}
	Finally, we state a probabilistic representation of probability measure-valued solutions. 
	\begin{kor}\label{cor:prob-repr}
		Suppose the assumptions of Theorem \ref{thm1} hold and, in addition, $\tilde{\omega}(t)\in \Pscr(\R^2)$ for all $t \geq 0$. Then there is a stochastic process $X = (X_t)_{t\geq 0}$ on a probability space $(\Omega, \Fscr, P)$ such that $P$-a.e. of its paths is an integral solution to \eqref{eq_intro} with $\tilde{\omega}$ instead of $\omega$, and $P\circ X^{-1} = \eta$. In particular, 
		$$P \circ X_t^{-1} = \tilde{\omega}(t),\quad \forall t \in [0,\infty).$$
	\end{kor}
	\begin{proof}
		With the additional assumption, $\eta$ from Theorem \ref{thm1} can be constructed as a probability solution of the nonlinear martingale problem with operator $(Lf)(t,x):= (K*\tilde{\omega}(t))(x)\cdot \nabla f(x)$. For every such $\eta$ there is a stochastic process $X$ on a probability space, which solves the corresponding (stochastic) differential equation with path law $\eta$, see \cite{Stroock87}. In the present case, this differential equation is deterministic and given by \eqref{eq_intro} with $\tilde{\omega}$ instead of $\omega$.
	\end{proof}
	
	\section{Nonlinear Markov processes associated with solutions to \texorpdfstring{$2D$}{2D} vorticity Euler equation}\label{sect:NL-MP}
	
	%
	
	\subsection{Nonlinear Markov processes}\label{subsect:NL-MP}
	We use the notation
	$$\Omega := C([0,\infty),\R^d), \quad \Fscr_t := \sigma(\pi_r,0\leq r \leq t).$$
	It is a classical result that the probability measure-valued solutions of a well-posed time-homogeneous linear FPE, i.e. \eqref{FPE} with coefficients only depending on $x$, are the transition kernel of a uniquely determined Markov process $\{\mathbb{P}_{x}\}_{x \in \R^d}$. This Markov process consists of the unique solution path laws $\mathbb{P}_{x}$ of the corresponding non-distribution dependent SDE \eqref{DDSDE}, with coefficients only depending on $X_t$. 
	
	This relation to Markov processes is not generally true for nonlinear equations \eqref{FPE} and distribution-dependent SDEs \eqref{DDSDE}, see \cite{R./Rckner_NL-Markov22}. In \cite{R./Rckner_NL-Markov22}, the following notion of nonlinear Markov processes was studied. We restrict the presentation here to the time-homogeneous case. 
	\begin{dfn}\label{def:NL-Markov-process}
		Let $\Pscr_0 \subseteq \Pscr$. A \textit{nonlinear Markov process} is a family $\{\mathbb{P}_{\zeta}\}_{\zeta\in\Pscr_0}$ of probability measures on $\Omega$ such that
		\begin{enumerate}
			\item[(i)] $\mathbb{P}_{\zeta}\circ\pi_t^{-1} =: \mu^{\zeta}_t\in \mathcal{P}_0$ for all $0 \leq t, \zeta \in \Pscr_0$.
			\item[(ii)] The \textit{nonlinear Markov property} holds: for all $0\leq s, t$, $\zeta \in \Pscr_0$ and $A\in \Bscr(\R^d)$
			\begin{equation}\label{Markov-prop}\tag{MP}
				\mathbb{P}_{\zeta}(\pi_{t+s} \in A|\mathcal{F}_{s})(\cdot) = p_{\mu^{\zeta}_s,\pi_s(\cdot)}(\pi_{t}\in A) \quad \mathbb{P}_{s,\zeta}-\text{a.s.},
			\end{equation}
			where $p_{\mu^{\zeta}_s,y}, y \in \R^d$, is a regular conditional probability kernel of $\mathbb{P}_{\mu^{\zeta}_s}[\,\,\cdot\, \,| \pi_0=y]$, $y \in \R^d$.
		\end{enumerate}
	\end{dfn}
	The difference to classical Markov processes is the dependence on $\mu^\zeta_s$ on the RHS of \eqref{Markov-prop}.
	The name \emph{nonlinear} Markov process is due to the fact that in Definition \ref{def:NL-Markov-process} the map $\Pscr_0 \ni \zeta \mapsto \mu^{\zeta}_t$ is, in general, not linear on its domain, even if $\Pscr_0$ is a convex set (which is not assumed). These processes share many features with classical Markov processes, for instance: they model memoryless evolutions in time, their path laws are uniquely determined by a family of one-dimensional kernels (but different from the classical case), and their marginals satisfy the \textit{flow property}
	\begin{equation}\label{eq:flow-prop}
		\mu^\zeta_t \in \Pscr_0,\quad \mu^{\zeta}_{t+s} = \mu^{\mu^{\zeta}_s}_t,\quad \forall0\leq s,t,\,\zeta \in \Pscr_0.
	\end{equation}
	this identity is not identical to the Chapman--Kolmogorov equations. Indeed, postulating the latter does not even make sense when $\Pscr_0$ does not contain all Dirac measures (which is usually the case for our results in this paper). But even when $\Pscr_0 = \Pscr$, \eqref{eq:flow-prop} implies the Chapman--Kolmogorov equations only if $\mu^\zeta_t = \int \mu^{\delta_x}_t \, d\zeta(x)$ for all $t \geq 0$ and $\zeta \in \Pscr$. But the latter need not hold ($x\mapsto \mu^{\delta_x}_t$ need not even be measurable).
	The class of nonlinear Markov processes contains the class of classical ones. For these claims and further details, we refer to \cite{R./Rckner_NL-Markov22}, which contains the following connection between \eqref{FPE}, \eqref{DDSDE} and nonlinear Markov processes. For a weakly continuous curve $\nu: [0,\infty)\to \Pscr$, denote by $M^{\zeta}_{\nu}$ the convex set of weakly continuous probability solutions $\mu$ to the \emph{linear} FPE
	$$	\partial_t \mu_t = \partial^2_{ij}\big(a_{ij}(\nu_t,x)\mu_t\big) - \partial_i \big(b_i(\nu_t,x)\mu_t\big),\quad \mu_0 = \zeta,\quad t \geq 0,$$
	such that $[(t,x)\mapsto a_{ij}(\nu_t,x)], [(t,x)\mapsto b_i(\nu_t,x)] \in L^1([0,T]\times \mathbb{R^d};d\mu_t dt)$ for all $T>0$, and let $M^{\zeta}_{\nu,\textup{ex}}$ be the set of extreme points of $M^{\zeta}_{\nu}$ (i.e. those elements $\nu^*$ from $M^\zeta_\nu$ for which $\nu^* = \alpha \nu^1 + (1-\alpha) \nu^2$ with $\alpha \in (0,1)$ implies $\nu^1 = \nu^2$ for any $\nu^1, \nu^2 \in M^\zeta_\nu$).
	
	\begin{theorem}[Thm.3.8 in \cite{R./Rckner_NL-Markov22}]\label{theorem1}
		Let $\Pscr_0 \subseteq\Pscr$ and 
		$$\{\mu^{\zeta}\}_{\zeta\in\Pscr_0},\quad \,\mu^{\zeta} = (\mu^{\zeta}_t)_{t\geq 0},\quad \mu^{\zeta}_0 = \zeta,$$
		be a probability measure-valued weakly continuous solution family to \eqref{FPE} with the flow property (called \emph{solution flow}), such that $\mu^{\zeta} \in M^{\zeta}_{\mu^{\zeta},\textup{ex}}$ for each $\zeta \in \Pscr_0$.
		
		Then the path laws $\mathbb{P}_{\zeta}$, $\zeta \in \Pscr_0$, of the unique weak solutions to \eqref{DDSDE} with one-dimensional time marginals $(\mu^{\zeta}_t)_{t \geq 0}$ form a nonlinear Markov process. 
	\end{theorem}
	In this case, each $\mathbb{P}_\zeta$ is the unique solution with marginals $(\mu^\zeta_t)_{t\geq 0}$ to the nonlinear martingale problem associated with $a_{ij}$ and $b_i$.
	
	The assertion contains an uniqueness result for weak solutions to \eqref{DDSDE}, which we do not discuss here (cf. \cite{R./Rckner_NL-Markov22}). The existence of a weak solution to \eqref{DDSDE} with marginals $(\mu^{\zeta}_t)_{t\geq 0}$ follows from the superposition principle, see Remark \ref{rem:SP-princ}. The following lemma provides a checkable criterion for the extremality condition in Theorem \ref{theorem1}.
	\begin{lem}[Lem.3.5 in \cite{R./Rckner_NL-Markov22}]\label{lem}
		$\mu^{\zeta} \in M^{\zeta}_{\mu^\zeta,\textup{ex}}$ holds if and only if $\mu^{\zeta}$ is the unique curve in $M^{\zeta}_{\mu^{\zeta}}\cap \Ascr_{\leq}(\mu^{\zeta})$, where
		$$\Ascr_{\leq}(\mu^{\zeta}) :=\big\{(\eta_t)_{t \geq 0}\in C([0,\infty),\Pscr): \eta_t \leq C \mu^{\zeta}_t \,\forall  t\geq 0 ,\textup{ for some }C>0\big\},$$
		where $C([0,\infty),\mathcal{P})$ is considered with the topology of weak convergence of measures.
	\end{lem}

	\subsection{Nonlinear Markov process associated with the \texorpdfstring{$2D$}{2D} vorticity Euler equation}
	We continue to use the notation $\Omega$ introduced at the beginning of the previous subsection, now for fixed dimension $d=2$.
	To apply the previous subsection to the $2D$ vorticity Euler equations, we consider \eqref{eq1} as the first-order nonlinear Fokker--Planck equation \eqref{eq2}. We have to identify subsets $\Pscr_0 \subseteq \Pscr$ such that there exists a family $\{\omega^\zeta\}_{\zeta \in \Pscr_0} \subseteq \Pscr_0$ of weakly continuous probability measure-valued solutions $\omega^\zeta = (\omega^\zeta_t)_{t\geq 0}$ to \eqref{eq1} with the flow property \eqref{eq:flow-prop}
	such that each $\omega^\zeta$ is the restricted unique (in the sense of Lemma \ref{lem}) solution to the linear equation
	\begin{equation}\label{aux567}
		\partial_t u + \divv(K(\omega^\zeta) u) = 0,\quad u(0)=\zeta.
	\end{equation}
	The associated differential equation is \eqref{eq_intro}, and a martingale solution to the latter is any Borel probability measure $Q$ on $\Omega$ concentrated on absolutely solution curves to \eqref{eq_intro} with $Q\circ \pi_t^{-1}$ instead of $\omega(t)$, see Remark \ref{rem:mg-sol-det-case}. The main result of this section is the following theorem.
	
	\begin{theorem}\label{thm:main-appl-selection}
		Let $p\geq 2$ and $\mathfrak{P} := \mathcal{P}\cap (L^1\cap L^p)(\R^2)$. There is a family of weakly continuous probability solutions $\{\omega^\zeta\}_{\zeta \in \mathfrak{P}}$, $\omega^\zeta_0 = \zeta$, to \eqref{eq1} in $L^\infty([0,\infty);(L^1\cap L^p)(\R^2))$ with the flow property, and a nonlinear Markov process $\{\mathbb{P}_\zeta\}_{\zeta \in \mathfrak{P}}$, consisting of probability measures $\mathbb{P}_\zeta$ on $\Omega$, such that $\mathbb{P}_\zeta$ is concentrated on absolutely continuous integral solutions to \eqref{eq_intro} with $\omega^\zeta$  instead of $\omega$, and
		\begin{equation*}
			\mathbb{P}_\zeta \circ \pi_t^{-1} = \omega^\zeta(t),\quad \forall t \geq 0.
		\end{equation*}
		Moreover, this nonlinear Markov process is uniquely determined by $\{\omega^\zeta\}_{\zeta \in \mathfrak{P}}$.
	\end{theorem}
	For the proof, we need the following preparations regarding the selection of a solution flow to \eqref{eq1}.
	
	\paragraph{Flow selections.}
	The following proposition is the key for the proof of Theorem \ref{thm:main-appl-selection} and might be of independent interest. It follows from the adaption to functions given in \cite{CarKap2020} of the deep idea of Krylov \cite{Kry1973}, see also \cite{StrVar1979}, developed for Markov families. Selection theorems for infinite dimensional systems have been later developed, see \cite{FlaRom2008,Rom2008,GolRocZha2009}, and more recently \cite{BreFeiHof2020a,BreFeiHof2020b,BreFeiHof2020c,CarHofNilRan2022,Rehmeier-nonlinear-flow-JDE}.
	
	\begin{prop}\label{prop:Romitos-prop}
		Let $p\geq2$. There exists a map
		\[
		\omega: [0,\infty)\times (L^1\cap L^p)(\R^2)
		\longrightarrow (L^1\cap L^p)(\R^2)
		\]
		such that for every $\zeta\in  (L^1\cap L^p)(\R^2)$, $(\omega^\zeta_t)_{t\geq 0} := (\omega(t,\zeta)(x)dx)_{t\geq 0}$ is a weakly continuous solution to \eqref{eq1} in $L^\infty([0,\infty);(L^1\cap L^p)(\R^2))$ with initial condition $\zeta$. Moreover, $\{\omega^\zeta\}_{\zeta \in (L^1\cap L^p)(\R^2)}$ has the flow property, and the sub-family $\{\omega^\zeta\}_{\zeta \in \mathfrak{P}}$ has the flow property in $\mathfrak{P}$.
	\end{prop}
	
	\begin{rem}
		Proposition \ref{prop:Romitos-prop} proves a slightly more general result than needed for Theorem \ref{thm:main-appl-selection}. We notice that a flow selection of probability solutions (i.e. the final part of the proposition) can directly be obtained from \cite[Proposition 4.8]{Rehmeier-nonlinear-flow-JDE}.
	\end{rem}
	
	\begin{rem}
		In the case $p<2$ (and, in light of our results and \cite{BruColDel2021}, $p>3/2$) unfortunately Proposition \ref{prop:p>2} does not hold, so we cannot guarantee uniform bounds on the vorticity in $L^1\cap L^p$ \emph{for all times}.
		
		If one is content with conservation laws (and thus the bounds on which Lemma \ref{l:compactmeasurable}, and in turn Proposition \ref{prop:Romitos-prop}, are based) that hold for almost all times (an idea initially introduced in \cite{FlaRom2008}), it is possible to obtain a selection for which the flow property is true only for almost all times. However, it is not clear how to turn this result into the existence of a nonlinear Markov process.
		
		Otherwise, one can rely on the approach developed in \cite{BreFeiHof2020a,BreFeiHof2020b}, which is based on including some conservation laws into the definition of solution. This way, possible time-discontinuities of the conservation laws are hard-coded into the solution and a selection with the flow property valid for all times can be obtained. Unfortunately, this would eventually require a Fokker-Planck equation \emph{also} for the conservation laws. At this stage, this appears unfeasible.
		
		The same remarks apply also to variants of Euler, such as the viscous/inviscid generalized SQG (see for instance \cite{Res1995,HelPieGarSwa1995,CorFefRod2004,ChaConCorGanWu2012} and references therein, and \cite{GelRom2021} for a statistical approach). We do not pursue this line of research here and plan to develop it in a future publication.
	\end{rem}
	
	For the proof of Proposition \ref{prop:Romitos-prop}, we need some auxiliary results.
	\begin{lem}\label{p:renormalized}
		Let $p\geq2$, and $\omega\in L^\infty(0,\infty;(L^1\cap L^p)(\R^2)$ be a weak solution of \eqref{eq2} with initial condition $\zeta$. Then the weakly continuous version of $\omega$ (cf. Lemma \ref{l:euler-vort-weak2}, here again denoted $\omega$) satisfies
		\begin{itemize}
			\item $\|\omega(t)\|_{L^q}=\|\zeta\|_{L^q}$, for every $t\geq0$ and every $q\in[1,p]$,
			\item $\int_{\R^2}\omega(t,x)\,dx = \int_{\R^2}\zeta(x)\,dx$, for every $t\geq0$,
			\item if $\zeta(x)\geq0$, {a.\,e.}, then $\omega(t,x)\geq0$, {a.\,e.} for every $t\geq0$.
		\end{itemize}
	\end{lem}
	\begin{proof}
		By \cite[Proposition 1]{L06_ES>1}, $\omega$ is a renormalized solution in the sense of \cite{DiPL89}. In particular, for every $\theta\geq0$, the Lebesgue measure of the sets
		\[
		\{x\in\R^2:\omega(t,x)>\theta\},\qquad
		\{x\in\R^2:\omega(t,x)<-\theta\},
		\]
		and thus of $\{x\in\R^2:|\omega(t,x)|>\theta\}$, is independent from $t$. All assertions follow from here.
	\end{proof}
	
	Let $\Sscr=C_\textup{w}([0,\infty);L^p(\R^2))$. For $\zeta\in (L^1\cap L^p)(\R^2)$, denote by $\Cscr(\zeta)$ the subset of $\Sscr$ of all weakly continuous weak solutions of \eqref{eq2} with initial condition $\zeta$.
	The proof of the next result is postponed to the appendix.
	
	\begin{lem}\label{l:compactmeasurable}
		Let $p\geq2$.
		For every sequence $(\zeta^n)_{n\geq1}$ converging to $\zeta$ in $(L^1\cap L^p)(\R^2)$, and for every sequence $\omega_n\in\Cscr(\zeta^n)$, $(\omega^n)_{n\geq1}$ has a limit point in $C_\textup{w}([0,\infty);L^p(\R^2))$, which belongs to $\Cscr(\zeta)$.
	\end{lem}
	
	Now we give the proof of Proposition \ref{prop:Romitos-prop}.
	
	\begin{proof}[Proof of Proposition \ref{prop:Romitos-prop}]
		Set for brevity $\Lscr_p=(L^1\cap L^p)(\R^2)$. In view of \cite[Theorem 4]{CarKap2020}, we need to prove that
		\begin{itemize}
			\item[s1.] for every $\zeta\in\Lscr_p$, $\Cscr(\zeta)$ is a non-empty compact subset of $\Sscr$,
			\item[s2.] $\zeta \mapsto\Cscr(\zeta)$ is measurable (with respect to the Hausdorff topology),
			\item[s3.] for every $\zeta\in\Lscr_p$ and $t\geq0$, if $\omega\in\Cscr(\zeta)$, then $\omega(t+\cdot)\in \Cscr(\omega(t))$.
			\item[s4.] for every $\zeta\in\Lscr_p$ and $t>0$, if $\omega\in\Cscr(\zeta)$ and $\omega'\in\Cscr(\omega(t))$, then $\omega\oplus_t\omega'\in\Cscr(\zeta)$, where $\omega\oplus_t\omega'(s):=\omega(s)\uno_{[0,t]}(s) + \omega'(s-t)\uno_{(t,\infty]}(s)$.
		\end{itemize}
		We first notice that properties (s3), (s4) are immediate when one considers the equivalent definition of weak solution given in Lemma \ref{l:euler-vort-weak2}. Compactness of $\Cscr(\zeta)$ follows from Lemma \ref{l:compactmeasurable} by taking $\zeta^n=\zeta$, while measurability follows again from Lemma \ref{l:compactmeasurable} and \cite[Lemma 12.1.8]{StrVar1979}. To complete the proof of (s1) it suffices to prove that each $\Cscr(\zeta)$ is non-empty. Existence of solutions in $L^\infty([0,\infty);(L^1\cap L^p)(\R^2))$ with initial condition in $\Lscr_p$ is classical, see for instance \cite{Che1995} and the references therein. By Lemma \ref{l:euler-vort-weak2} (iii) each such solution has a version in $C(\zeta)$.
		
		The final part of the assertion follows, since by Lemma \ref{p:renormalized} solutions in each $C(\zeta)$, $\zeta \in (L^1\cap L^p)(\R^2)$, have constant $L^1$-norm and are nonnegative, if $\zeta$ is nonnegative.
	\end{proof}
	
	\begin{proof}[Proof of Theorem \ref{thm:main-appl-selection}]
		We aim to apply Theorem \ref{theorem1}.
		By Proposition \ref{prop:Romitos-prop}, there is a family $\{\omega^\zeta\}_{\zeta \in \mathfrak{P}}$ of weakly continuous probability solutions to \eqref{eq1} with $\omega_0^\zeta = \zeta$ with the flow property in $\mathfrak{P}$. Since $\omega^\zeta \in L^\infty([0,\infty);(L^1\cap L^p)(\R^2))$, by Lemma \ref{lem:K-prop} (iii) we have $\omega^\zeta v^\zeta \in L^\infty([0,\infty);L^1(\R^2,\R^2))$, so $\omega^\zeta \in M^\zeta_{\omega^\zeta}$ (we set $v^\zeta := K *w^\zeta)$. It remains to prove one of the equivalent assumptions of Lemma \ref{lem} for each $\omega^\zeta$. To this end note that, using the notation of Lemma \ref{lem} adapted to the present situation, $M^\zeta_{\omega^\zeta}\cap \Ascr_{\leq}(\omega^\zeta)$ is a subset of all solutions to \eqref{aux567} in $L^\infty([0,\infty);L^p(\R^2))$ with initial datum $\zeta$. Since $p\geq 2$ implies $p'\leq p$, Lemma \ref{lem:K-prop} yields $v^\zeta \in L^\infty([0,\infty);W^{1,p'}_{\textup{loc}}(\R^2,\R^2))$, and together with Lemma \ref{lem:K-prop} (i) this allows to apply \cite[Thm.II.2]{DiPL89} to deduce uniqueness of solutions to \eqref{aux567} in $L^\infty([0,\infty),L^p(\R^2))$ with initial datum $\zeta$. Therefore, the proof is complete.
	\end{proof}

	\subsection{Nonlinear Markov processes in regularity persistent spaces}
	
	Here we report some results in the literature on spaces $E$ with the property that for any initial vorticity in $E$, there exists a unique weakly continuous weak solution for \eqref{eq1} with values in $E$ at all times (called \emph{persistence of regularity}). In general, persistence of regularity is not obvious. For instance the results of \cite{Vis1999} fail to identify such spaces $E$. Each of the spaces
	\begin{equation}\label{aux11}
		(L^1\cap L^\infty)(\R^2)\subset Y^\Theta\subset L^1(\R^2)\cap Y^\Theta_\uloc
	\end{equation}
	is an example of spaces $E$ allowing for regularity persistent solutions for every initial vorticity. Here, for a non-decreasing function $\Theta: [1,\infty) \to (0,\infty)$, $Y^\Theta$ is the well-known \emph{Yudovich space},
	$$Y^\Theta := \bigg\{ f \in \bigcap_{p\in [1,\infty)} L^p(\R^2): \sup_{p \geq 1} \frac{||f||_{L^p}}{\Theta(p)} < \infty\bigg\},$$
	while $Y^\Theta_\uloc$ denotes the \emph{uniformly localized Yudovich space}
	\[
	Y_\uloc^\Theta
	:= \Bigl\{f\in\bigcap_{p\in[1,\infty)}L^p_\uloc(\R^2): \sup_{p\geq 1}\frac{\|f\|_{L^p_\uloc}}{\Theta(p)}<\infty\Bigr\},
	\]
	where 
	$$L^p_\uloc(\R^2) := \bigg\{ f \in L^p_\loc(\R^2) :  ||f||_{L^p_\uloc} = \sup_{x\in \R^2} ||f||_{L^p(B_1(x))} < \infty \bigg\}.$$
	The 2D vorticity Euler equations in $(L^1\cap L^\infty)(\R^2)$ were studied in \cite{Yud1963,Che1995}. The spaces $Y^\Theta$ and $Y^\Theta_\uloc$ have been studied in \cite{Yud1995} and \cite{Tan2004,CS21}, respectively. We refer to these works for the well-posedness results stated below.
	Regarding $Y^\Theta_\uloc$, we follow \cite{CS21} by considering $\varphi_\Theta:[0,\infty)\to[0,\infty)$ defined via 
	$$\varphi_\Theta(0)=0, \,\,\,\varphi_\Theta(r)=r(1-\log r)\Theta(1-\log r) \textup{ for }r\in[0,e^{-2}],\,\, \varphi_\Theta = \varphi_\Theta(e^{-2}) \text{ for }r\geq e^{-2}.$$
	Denote by $C_b^{\varphi_\Theta}(\R^2,\R^2)$ the continuous bounded vector fields with modulus of continuity $\varphi_\Theta$, i.e.
	$$C_b^{\varphi_\Theta}(\R^2,\R^2) := \bigg\{g \in C_b(\R^2,\R^2) : \sup_{x\neq y}\frac{|g(x)-g(y)|}{\varphi_\Theta(|x-y|)} < \infty\bigg\},$$
	with norm
	\begin{equation*}\label{eq:Cphi}
		\|g\|_{C_b^{\varphi_\Theta}}
		= \|g\|_{L^\infty} + \sup_{x\neq y}\frac{|g(x)-g(y)|}{\varphi_\Theta(|x-y|)}.
	\end{equation*}

	\begin{theorem}[\cite{CS21}]\label{t:yudovichclass}
		Assume that $\varphi_\Theta$ is concave on $[0,\infty)$ and
		\begin{equation}\label{eq:preosgood}
			\int_1^\infty \frac{1}{p\Theta(p)} dp
			= \infty.
		\end{equation}
		Let $\zeta\in L^1(\R^2)\cap Y_\uloc^\Theta$. Then the class of vaguely continuous weak solutions to \eqref{eq1} in $L^\infty_\loc([0,\infty);L^1(\R^2)\cap Y^\Theta_\uloc)$ with initial condition $\omega_0$ contains exactly one element $\omega^\zeta$. Moreover, if $\zeta \geq 0$ and $||\zeta||_{L^1} = c$, then $\omega^\zeta \geq 0$ and $||\omega^\zeta(t)||_{L^1} = c$  for all $t >0$. 
		In addition, $v^\zeta = K*\omega^\zeta$ belongs to  $L^\infty_\loc([0,\infty);C_b^{\varphi_\Theta}(\R^2,\R^2))$.
	\end{theorem}
	
	\begin{rem}
		\begin{enumerate}
			\item [(i)] In the situation of Theorem \ref{t:yudovichclass}, it is known that if either $\zeta \in (L^1\cap L^\infty)(\R^2)$ or $\zeta \in Y^\Theta$, then $\omega^\zeta \in L^\infty_\loc([0,\infty);(L^1\cap L^\infty)(\R^2))$ or $\omega^\zeta \in L^\infty_\loc([0,\infty);L^1(\R^2)\cap Y^\Theta)$, respectively. Hence each of the spaces from \eqref{aux11} allows for unique regularity-persistent solutions.
			\item [(ii)]  By \eqref{eq:preosgood} the function $\varphi_\Theta$ satisfies the \emph{Osgood condition},
			\[
			\int_0^1\frac1{\varphi_\Theta(x)}\,dx
			=\infty.
			\]
			Thus, by the last part of Theorem \ref{t:yudovichclass}, the ODE associated with the vector field $K*\omega^\zeta = v^\zeta$ has unique solutions for every initial value. This implies in particular that $\omega^\zeta$ is the only weakly continuous weak solution $w$ to the linear equation \eqref{aux567} in $L^\infty_\loc([0,\infty);L^1(\R^2))$ such that $w v^\zeta \in L^1([0,T]\times \R^2,\R^2; dx dt)$ for all $T>0$. Thus, in this situation the equivalent conditions of Lemma \ref{lem} are satisfied.
		\end{enumerate}
	\end{rem}
	Altogether, by Theorem \ref{theorem1}, Lemma \ref{lem}, Theorem \ref{t:yudovichclass} and the previous remark, we obtain the following result.
	\begin{prop}\label{prop:well-posed-case}
		Let $E$ be either of the spaces from \eqref{aux11} and assume $\varphi_\Theta$ is concave and satisfies \eqref{eq:preosgood}. Then there exists a uniquely determined nonlinear Markov process $\{\mathbb{P}_\zeta\}_{\zeta \in \Pscr \cap E}$ with one-dimensional time marginals given by the family of unique solutions $\omega^\zeta$ to \eqref{eq1} in $L^\infty_\loc([0,\infty),E)$ with initial datum $\zeta$. Each $\mathbb{P}_\zeta$ is a probability measure on $\Omega$ concentrated on absolutely continuous integral solutions to \eqref{eq_intro} with $\omega^\zeta$ instead of $\omega$.
	\end{prop}
	Since for $E = (L^1\cap L^\infty)(\R^2)$ and $E = L^1(\R^2) \cap Y^\Theta$ we have $E \subseteq (L^1\cap L^p)(\R^2)$, in these cases the nonlinear Markov process is a subfamily of the process obtained in Theorem \ref{thm:main-appl-selection}.
	
	\begin{rem}
		Results of persistence of regularity in BMO-like spaces can be found in \cite{BerKer2014} (see also \cite{Vis1999,BerHmi2015}) and \cite{CheMiaZhe2019}. One reason why BMO spaces are relevant in this setting is that by interpolation results between Lebesgue and BMO spaces \cite[Chapter 7]{Gra2009}, it follows that given $p\in(1,\infty)$, for every $q\in(p,\infty)$ there is $C(q)>0$ such that for $f\in L^p(\R^2)\cap BMO$,
		\[
		\|f\|_{L^q}
		\leq C(q)\|f\|_{L^p}^{p/q}\|f\|_{BMO}^{1-\frac{p}{q}}.
		\]
		In other words $\sup_q\|f\|_{L^q}/C(q)$ is finite, and thus $L^p(\R^2)\cap BMO$ is included in a suitable Yudovich space.
	\end{rem}

	\appendix
	\section{Appendix}
	We complete the work with the proof of Lemma \ref{l:compactmeasurable}.
	
	\begin{proof}[Proof of Lemma \ref{l:compactmeasurable}]
		First, let $p>2$. Let $(\zeta^n)_{n\geq1}$ and $(\omega^n)_{n\geq1}$ be sequences as in the statement of the proposition. By Lemma \ref{p:renormalized} we have
		\begin{equation}\label{eq:cm_boundomega}
			\sup_n\sup_{t\in[0,\infty)}\|\omega^n(t)\|_{L^q}
			\leq \sup_n\|\zeta^n\|_{L^q}
			\leq \sup_n(\|\zeta^n\|_{L^1} + \|\zeta^n\|_{L^p})
			\vcentcolon= C_0,
		\end{equation}
		for all $q\in[1,p]$. Moreover, if $v^n=K * \omega^n$, then by standard properties of the Biot-Savart kernel,
		\begin{equation}\label{eq:cm_boundv}
			\sup_n\sup_{t\in [0,\infty)}\|v^n(t)\|_{W^{1,p}}
			\lesssim C_0.
		\end{equation}
		We prove now that $(\omega^n)_{n\geq1}$ is uniformly equicontinuous for the weak topology of $L^p(\R^2)$. First, by definition of weak solution and by \eqref{eq:cm_boundomega} and \eqref{eq:cm_boundv}, if $\phi\in C^\infty_c(\R^2)$, then
		\begin{equation}\label{eq:cm_lipschitz}
			\begin{multlined}[.9\linewidth]
				\Big|\int_{\R^2}\omega^n(t,x)\phi(x)\,dx - \int_{\R^2}\omega^n(s,x)\phi(x)\,dx\Big|\leq\\
				\leq \|\omega^n\|_{L^\infty(L^1)}\|v^n\|_{L^\infty(L^p)}\|\nabla\phi\|_{L^\infty}|t-s|
				\lesssim C_0^2\|\nabla\phi\|_{L^\infty} |t-s|
			\end{multlined}
		\end{equation}
		for every $s,t\geq0$. Since compactly supported smooth functions are dense in $L^{p'}(\R^2)$, where $p'$ is the H\"older conjugate exponent of $p$, we have uniform equicontinuity in $L^p$ with the weak topology.
		%
		%
		Indeed, if $\phi\in L^{p'}(\R^2)$, $\varepsilon>0$, $\phi_\varepsilon\in C^\infty_c(\R^2)$ is such that $\|\phi-\phi_\varepsilon\|_{L^{p'}}\leq\varepsilon/(4C_0)$ and $\delta=\varepsilon/(2C_0^2\|\nabla\phi_\varepsilon\|_{L^\infty})$, then for $s,t\geq0$ with $|t-s|\leq\delta$,
		\[
		\begin{aligned}
			\Big|\int_{\R^2}\omega^n(t)\phi\,dx - \int_{\R^2}\omega^n(s)\phi\,dx\Big|
			&\leq \Big|\int_{\R^2}\omega^n(t)\phi_\varepsilon\,dx - \int_{\R^2}\omega^n(s)\phi_\varepsilon\,dx\Big|\\
			&\quad  + \Big|\int_{\R^2}(\omega^n(t)-\omega^n(s))(\phi-\phi_\varepsilon)\,dx\Big|\\
			&\leq C_0^2\|\nabla\phi_\varepsilon\|_{L^\infty}|t-s|
			+ 2C_0\|\phi-\phi_\varepsilon\|_{L^{p'}}\\
			&\leq\varepsilon.
		\end{aligned}
		\]
		%
		%
		Thus, $(\omega_n)_{n\geq1}$ is uniformly equicontinuous in $L^p(\R^2)$ for the weak convergence, and uniformly bounded in $L^p(\R^2)$. Since bounded sets in $L^p(\R^2)$ are sequentially compact for the weak topology, by the Ascoli-Arzelà theorem and a diagonal argument, there is $\omega\in C_\textup{w}([0,\infty);L^p(\R^2))$ such that for every $T>0$ and $\phi\in L^{p'}(\R^2)$, along a subsequence (again denoted $\omega^n$)
		\[
		\sup_{t\in[0,T]}\Big|\int_{\R^2}\omega^n(t)\phi\,dx - \int_{\R^2}\omega(t)\phi\,dx\Big|
		\longrightarrow 0,
		\]
		as $n\to\infty$. In particular, $\omega\in L^\infty(0,\infty);L^p(\R^2))$.
		%
		%
		Let us prove $\omega\in L^\infty(0,\infty;L^1(\R^2))$. Let $R>0$ and $\eta_R=\uno_{B_R(0)}$. Since $\omega^n(t)\to\omega(t)$ weakly in $L^p(\R^2)$, we have $\omega^n(t)\eta_R\to\omega(t)\eta_R$ weakly in $L^1(\R^2)$, thus
		\[
		\int_{\R^2}\omega(t)\eta_R\,dx
		\leq\liminf_n\int_{\R^2}|\omega^n(t)\eta_R|\,dx
		\leq C_0,
		\]
		and the $L^1$-bound follows by monotone convergence as $R\uparrow\infty$.
		%
		%
		We turn to convergence of velocities. We have, for $\phi\in C^\infty_c(\R^2)$ and $t>0$,
		\begin{equation}\label{eq:cm_bs}
			\int_{\R^2}v^n(t)\phi\,dx
			= - \int_{\R^2}\omega^n K * \phi\,dx
			\longrightarrow - \int_{\R^2}\omega K * \phi \,dx
			= \int \phi K * \omega\,dx.
		\end{equation}
		Set $v = K * \omega$. By \eqref{eq:cm_boundv} and Morrey's inequality, $(v^n(t))_{n\geq1}$ is uniformly bounded and uniformly equi-continuous for all $t>0$. By Ascoli-Arzelà's theorem, for all $t$ there is a subsequence (dependent on $t$, so far) converging uniformly on compact subsets of $\R^2$. By \eqref{eq:cm_bs}, the limit is $v(t)$, thus the whole sequence $v^n(t)\to v(t)$ converges uniformly on compact subsets of $\R^2$.
		
		We can finally prove that $\omega$ is a weak solution. Let $\phi\in C^\infty_c(\R^2)$, then for every $t$,
		\[
		\int_{\R^2}\omega^n(t,x)\phi(x)\,dx
		\longrightarrow \int_{\R^2}\omega(t,x)\phi(x)\,dx.
		\]
		Moreover, since $\omega^n\to\omega$ in $C_\textup{w}([0,\infty);L^p(\R^2))$ and $v^n(s)\to v(s)$ uniformly on compact subsets of $\R^2$, we have for all $s>0$
		\[
		\int_{\R^2}\omega^n(s,x)v^n(s,x)\cdot\nabla\phi(x)\,dx
		\longrightarrow\int_{\R^2}\omega(s,x)v(s,x)\cdot\nabla\phi(x)\,dx.
		\]
		By \eqref{eq:cm_boundomega}, \eqref{eq:cm_boundv} and the dominated convergence theorem, the time integral converges as well.
		\smallskip
		
		We turn to the case $p=2$ and point out how to adapt the argument from $p>2$. Bound \eqref{eq:cm_boundomega} still holds (with $p=2$), while \eqref{eq:cm_boundv} is replaced by
		\begin{equation}\label{eq:cm_boundv2}
			\sup_n\sup_{t\in[0,\infty)}\bigl(\|v^n(t)\|_{L^r} + \|\nabla v^n(t)\|_{L^2}\bigr)
			\lesssim C_0,
		\end{equation}
		for every $r\in(2,\infty)$, and \eqref{eq:cm_lipschitz} is replaced by
		\begin{equation}\label{eq:cm_lipschitz2}
			\begin{multlined}[.9\linewidth]
				\Big|\int_{\R^2}\omega^n(t,x)\phi(x)\,dx - \int_{\R^2}\omega^n(s,x)\phi(x)\,dx\Big|\leq\\
				\leq \|\omega^n\|_{L^\infty(L^q)}\|v^n\|_{L^\infty(L^r)}\|\nabla\phi\|_{L^\infty}|t-s|
				\lesssim C_0^2\|\nabla\phi\|_{L^\infty} |t-s|,
			\end{multlined}
		\end{equation}
		with $\frac1q+\frac1r=1$, $q<2$ and $r>2$. We thus get uniform equicontinuity in $L^2$ with the weak topology, and $(\omega^n)_{n\geq1}$ converges uniformly on compact intervals to some $\omega\in C_\textup{w}([0,\infty);L^2(\R^2))$. With similar arguments as in the case $p>2$ we deduce $\omega\in L^\infty(0,\infty);L^1(\R^2))$.
		
		For the convergence of velocities we first notice that \eqref{eq:cm_bs} holds and set $v=K * \omega$. Indeed, $w^n(t) \longrightarrow \omega(t)$ weakly in $L^2(\R^2)$ and $(w^n(t))_{n\in \N}$ is bounded in each $L^q(\R^2)$, $q \in [1,2]$. Thus $w^n \longrightarrow \omega$ weakly in each $L^q(\R^2)$, $q \in (1,2]$. Since $K*\phi \in L^q(\R^2,\R^2)$ for all $q>2$, \eqref{eq:cm_bs} follows. This time we prove that $(v^n(t))_{n\geq1}$ converges strongly to $v(t)$ in $L^q_\loc(\R^2,\R^2)$ for all $t>0$ and every $q\in(2,\infty)$. Let $q>2$ and $R>0$, and set $B=B_R(0)$, then, by \eqref{eq:cm_boundv2}, $(v^n(t))_{n\geq1}$ is compact in $L^q(B,\R^2)$ for all $t$. Therefore there is a subsequence (again, in principle, depending on $t$) converging to $v$, due to \eqref{eq:cm_bs}. We can conclude that $v^n(t)$ converges strongly in $L^q(B,\R^2)$ to $v(t)$ for all $t>0$. This is sufficient to pass to the limit in order to prove that $\omega$ is a weak solution, as in the case $p>2$.
	\end{proof}
	
	\paragraph{Acknowledgment.}
	The first author is funded by the Deutsche Forschungsgemeinschaft (DFG, German Research Foundation) – Project-ID 317210226 – SFB 1283.
	
	The second author acknowledges the partial support of the project PNRR - M4C2 - Investimento 1.3, Partenariato Esteso PE00000013 - \emph{FAIR - Future Artificial Intelligence Research} - Spoke 1 \emph{Human-centered AI}, funded by the European Commission under the NextGeneration EU programme, of the project \emph{Noise in fluid dynamics and related models} funded by the MUR Progetti di Ricerca di Rilevante Interesse Nazionale (PRIN) Bando 2022 - grant 20222YRYSP, of the project \emph{APRISE - Analysis and Probability in Science} funded by the the University of Pisa, grant PRA\_2022\_85, and of the MUR Excellence Department Project awarded to the Department of Mathematics, University of Pisa, CUP I57G22000700001.
	
	\bibliographystyle{plain}
	\bibliography{rr24euler}
\end{document}